\documentclass[letter]{article}
\usepackage{amssymb,amsfonts, amsmath, amsthm}
\usepackage[pdftex,colorlinks,linkcolor=blue,citecolor=red]{hyperref}
\usepackage[english]{babel}
\usepackage[latin1]{inputenc}
\usepackage{enumerate}

\newtheorem{teo}{Theorem}
\newtheorem{pro}{Proposition}
\newtheorem{defi}{Definition}
\newtheorem{cor}{Corolary}
\newtheorem{ejem}{Example}
\newtheorem{nota}{Remark}

\begin{document}

\title{Some generalizations of the notion of Lie algebra }
\author{
Dennise Garc\'ia-Beltr\'an and Jos\'e A. Vallejo\\
Facultad de Ciencias \\
Universidad Aut\'onoma de San Luis Potos\'i­ \\
{\normalsize Lat. Av.\ Salvador Nava \indent s/n,
78290-San Luis Potos\'i­ (M\'exico) }\\
{\small \emph{E-mail:}
\texttt{dennise.gb@alumnos.uaslp.edu.mx, jvallejo@fc.uaslp.mx }}
}
\date{July 26 2010}
\maketitle

\begin{abstract}
We introduce the notion of left (and right) quasi-Loday algebroids and a ``universal space'' for them, called a
left (right) omni-Loday algebroid, in such a way that Lie algebroids, omni-Lie algebras and omni-Loday algebroids are particular
substructures.
\end{abstract}

\section{Introduction}
There are several ways in which Lie algebras can be generalized. Recall that if $M$ is an 
$\mathcal{R}$-module\footnote{Here $\mathcal{R}$ is a commutative ring with unit element $1_{\mathcal{R}}$.}
endowed with a $\mathcal{R}$-bilinear bracket $[\;,\;]:M\times M\rightarrow M$ such that for all $u,v,w\in M$
\begin{enumerate}
\item[(i)] $[u,v]=-[v,u]$ (antisymmetry)
\item[(ii)] \label{ii} $[u,[v,w]]+[w,[u,v]]+[v,[w,u]]=0$ (Jacobi identity),
\end{enumerate}
then $(M,[\;,\;])$ is a Lie algebra structure over $M$. When $\mathcal{R}=\mathbb{R}$ (resp. $\mathbb{C}$) we speak about a real Lie algebra 
(resp. a complex Lie algebra).\\
Let us briefly mention some of these generalizations:
\begin{enumerate}[a)]
\item\label{gena} We can lift the restriction of antisymmetry. We then get the notion of Loday (or Leibniz\footnote{We prefer the 
denomination
``Loday algebra'', see \cite{Kosmann2} for an explanation.}) algebra \cite{Loday,Kosmann2}. More precisely, the pair 
$(M,[\;,\;])$ is a left 
Loday algebra if, instead of conditions (i), (ii) above, it satisfies the left Leibniz identity:
\begin{enumerate}
\item[(iii)]\label{lli} $[u,[v,w]]=[[u,v],w]+[v,[u,w]]$.
\end{enumerate}
Note that this condition can be expressed by saying that $[u,\;]$ is a derivation with respect to the product $[\;,\;]$.
Analogously, we can define right Loday algebras over $M$, by imposing that $[\;,w]$ be a derivation with respect to $[\;,\;]$:
\begin{enumerate}
\item[(iii')] $[[u,v],w]=[[u,w],v]+[u,[v,w]]$
\end{enumerate}
Note that any one of (iii) or (iii') is equivalent to the Jacobi identity when $[\;,\;]$ is antisymmetric.
\item\label{genb} Also, it is possible to consider a family of Lie algebras parametrized by points on a manifold $M$ which, with some natural 
geometric assumptions, lead to the idea of Lie algebroid introduced by J. Pradines (see \cite{Kosmann,Mackenzie}). To be precise, a Lie 
algebroid over a manifold $M$ (assume it real for simplicity) is given by a vector bundle  $\pi:E\rightarrow M$, an $\mathbb{R}$-bilinear 
bracket $[\;,\;]:\Gamma E\times \Gamma E\rightarrow \Gamma E$ defined on the $\mathcal{C}^{\infty}(M)$-module of sections of $E$, and a 
mapping $q_{E}:\Gamma E\rightarrow\mathfrak{X}(M)$ (called the anchor map) such that, for all $X,Y\in\Gamma E,\;f\in\mathcal{C}^{\infty}(M)$:
\begin{enumerate}[(1)]
\item $(\Gamma E,[\;,\;])$ is a real Lie algebra,
\item\label{LieAlg1} $[X,fY]=f[X,Y]+q_E(X)(f).Y$
\end{enumerate}
Note that, in the case when $M$ reduces to a single point, a Lie algebroid over $M=\{*\}$ is just a Lie algebra.
A basic property of Lie algebroids is that the anchor map $q_E$ is a Lie algebra morphism  when the bracket on $\mathfrak{X}(M)$ is taken as 
the Lie bracket of vector fields \cite{KM,Grabowski}.
\item\label{genc} Finally, A. Weinstein introduced in \cite{Weinstein} the concept of omni-Lie algebra, a structure that can be thought of as 
a kind of ``universal space" for Lie algebras: take any natural number $n\geq 2$ and consider the space product 
$\mathcal{E}_n=\mathfrak{gl}_n\times \mathbb{R}^n$ endowed with the $\mathbb{R}$-bilinear form 
$\{\;,\;\}:\mathcal{E}_n\times \mathcal{E}_n\rightarrow \mathcal{E}_n$ given by
\begin{equation*}
\{(A,x),(B,y)\}=\left([A,B],\frac{1}{2}\left(Ay-Bx\right)\right).
\end{equation*}
where $[A,B]=A\circ B-B\circ A$ is the Lie bracket of $\mathfrak{gl}_n$. Then, $(\mathcal{E}_n,\{\;,\;\})$ is the $n$-dimensional omni-Lie 
algebra. The reason behind this denomination is that any $n$-dimensional real Lie algebra $\mathfrak{g}$ is a closed maximal subspace of 
$(\mathcal{E}_n,\{\;,\;\})$.
\end{enumerate}

Our goal is to define a structure for which the constructions mentioned in \ref{gena}),\ref{genb}),\ref{genc}) appear as particular cases. In 
an absolutely unimaginative way, we will call it a left omni-Loday algebroid (of course, there exists the corresponding ``right" definition). 
As we will see, this also include as a particular case the notion of omni-Lie algebroid. Actually, the object we will construct will carry on a
bracket that has already appeared in the literature, although under a different approach. In the paper \cite{Kinyon1}, M. K. Kinyon and A. Weinstein
attacked the problem of integrating (in the sense of S. Lie's ``Third Theorem'') a Loday algebra\footnote{For more recent results in this topic,
called the coquecigrue problem, see \cite{Kinyon2,Ongay,Sheng2,Sheng3} and references therein.}, and they gave the following example: Take
$(\frak h ,[\; ,\; ])$ a Lie algebra, and let $V$ be an $\frak h -$module with left action on $V$ given by $(\zeta ,x)\mapsto \zeta x$.
Then, we have the induced left action of $\frak h$ on $\frak h \times V$
$$
\zeta (\eta ,y)=([\zeta ,\eta ],\zeta y).
$$
A binary operation $\cdot$ can be defined on $\mathcal E =\frak h \times V$ through
$$
(\zeta ,x)\cdot (\eta ,y)=\zeta (\eta ,y)=([\zeta ,\eta ],\zeta y).
$$
It turns out that $(\mathcal E ,\cdot )$ is a Loday algebra, and if $\frak h$ acts nontrivially on $V$, then $(\mathcal E ,\cdot )$ is \emph{not} a Lie
algebra. Kinyon and Weinstein called $\mathcal E$ with this Loday algebra structure the hemisemidirect product of $\frak h$ with $V$. Our omni-Loday 
algebroid will be a particular case of this construction, taking $\mathfrak{gl}(V)$ as $\frak h$ (see Definition \ref{d4}).

To achieve our goal, let us note that it is necessary to 
recast the definition of a Lie algebroid in  a form more suitable to an algebraic treatment, as in \ref{gena}), \ref{genc}). This can be easily done, just note 
that $\mathcal{C}^{\infty}(M)$ can be replaced by any $\mathcal{R}$-algebra $\mathcal{A}$, with unit element $1_{\mathcal{A}}$ and 
commutative, $\Gamma E$ by a faithful $\mathcal{A}$-module $\mathcal{F}$, and $\mathfrak{X}(M)$ by the module of derivations 
$Der_{\mathcal{R}}(\mathcal{A})$.

This idea was cleverly exploited by J. Grabowski who, in the paper \cite{Grabowski}, used it to prove the property of the anchor map
of being a Lie algebra morphism. In the same paper, it is proved that there exist obstructions to the existence of Loday algebroid structures
on vector bundles over a manifold $M$, stated in terms of the rank of these bundles (see Theorems \ref{teograb1}, \ref{teograb2} below). As
we will see, we can bypass these obstructions by considering left and right structures separately.
 
\textbf{Acknowledgements} The authors want to express their gratitude to Prof. M. K. Kinyon and Prof. Y. Sheng for their useful comments and for
providing references to previous work on Loday algebras. 

\section{Quasi-derivations}
The basic properties of a Lie algebroid are encoded in its anchor map, which in this context is a mapping 
$\rho:\mathcal{F}\rightarrow Der_{\mathcal{R}}(\mathcal{A})$.
We will assume that $\mathcal{F}$ is endowed with an $\mathcal{R}$-bilinear bracket $[\![\;,\;]\!]$, then $\rho$ is determined by two 
adjoints maps $ad^L_{\mathcal{A}},\;ad^R_{\mathcal{A}}:\mathcal{F}\rightarrow End_{\mathcal{R}}(\mathcal{F})$, given respectively by 
$ad^L_{\mathcal{A}}(X)=[\![X,\;]\!],\;ad^R_{\mathcal{A}}(X)=[\![\;,X]\!]$.\\
Under certain mild conditions, these mappings are quasi-derivations of $\mathcal{F}$, a property which is basic in the study of $\rho$. For 
instance, the fact that $ad^L_{\mathcal{A}},\;ad^R_{\mathcal{A}}$ are quasi-derivations allows one to prove that $\rho$ is a morphism of Lie 
algebras (when $(\mathcal{F},[\![\;,\;]\!])$ is Lie and we take the commutator of endomorphisms as the bracket on $End_{\mathcal{R}}
(\mathcal{F})$), see \cite{Grabowski} (we refer the reader to that paper for the proof of the results stated in this section).\\
We recall that an operator $D\in End_{\mathcal{R}}(\mathcal{F})$ is a quasi-derivation if for a given $f\in\mathcal{A}$ there exists 
$g\in\mathcal{A}$ such that
$$[D,\mu_f]=\mu_g,$$
where $[\;,\;]$ is the commutator of endomorphisms of $\mathcal{F}$, and $\mu_h(X)=h.X$, for any $h\in \mathcal{A},\;X\in\mathcal{F}$.
A quasi-derivation is called a tensor operator when 
$$[D,\mu_f]=0,\;\;\;\forall f\in \mathcal{A}.$$
Note that this is equivalent to $D$ being $\mathcal{A}$-linear (and not just $\mathcal{R}$-linear).
Some other straightforward properties of quasi-derivations are:
\begin{enumerate}[(1)]
\item The set of all the quasi-derivations of $\mathcal{F},\;QDer_{\mathcal{R}}(\mathcal{F})$ is an $\mathcal{R}$-module.
\item The commutator of endomorphisms on $Der_{\mathcal{R}}(\mathcal{F})$ restricts to a closed bracket on $QDer_{\mathcal{R}}(\mathcal{F})$ 
(i.e, if $D_1,\;D_2\in QDer_{\mathcal{R}}(\mathcal{F}), [D_1,D_2]\in QDer_{\mathcal{R}}(\mathcal{F})$). 
Thus $(QDer_{\mathcal{R}}(\mathcal{F}),[\;,\;])$ inherits the Lie algebra structure of $(End_{\mathcal{R}}(\mathcal{F}),[\;,\;])$.
\item $QDer_{\mathcal{R}}(\mathcal{F})$ is not just an $\mathcal{R}$-module. Defining, for any $f\in\mathcal{A}$ and 
$D\in QDer_{\mathcal{R}}(\mathcal{F})$,
$$f.D:=\mu_f\circ D,$$
it results that $QDer_{\mathcal{R}}(\mathcal{F})$ is an $\mathcal{A}$-module too.
\item\label{imp} $(QDer_{\mathcal{R}}(\mathcal{F}),[\;,\;])$ is not just a Lie algebra, but also a Poisson algebra (with the product given by 
the composition of endomorphisms).
\end{enumerate}
The following results will be crucial in the sequel.
\begin{teo}\label{t1}
There exists an $\mathcal{R}$-linear mapping  $\;\widehat{}:QDer_{\mathcal{R}}(\mathcal{F})\rightarrow Der_{\mathcal{R}}(\mathcal{A})$ 
such that
$$[D,\mu_f]=\mu_{\widehat{D}(f)}$$
\end{teo}
\begin{cor}\label{c1}
The $\mathcal{R}$-linear mapping $\;\widehat{}$ extends to a Lie algebra morphism:
$$[\widehat{D_1},\widehat{D_2}]=\widehat{[D_1,D_2]},\;\;\;\forall D_1,\;D_2\in QDer_{\mathcal{R}}(\mathcal{F}).$$
\end{cor}
Combining \eqref{imp} above with Theorem \ref{t1}, we also get:
\begin{cor}\label{c2}
If $\mathcal{A}$ is an $\mathcal{R}$-commutative algebra, the commutator $[\;,\;]$ on $QDer_{\mathcal{R}}(\mathcal{F})$ satisfies
\begin{equation*}
[D_1,f.D_2]=f.[D_1,D_2]+\widehat{D_1}(f).D_2
\end{equation*}
for all $D_1,\;D_2\in QDer_{\mathcal{R}}(\mathcal{F}),\;f\in \mathcal{A}$
\end{cor}

\section{Left Loday quasi-algebroids}
The formula obtained in Corollary \ref{c2} looks very similar to condition \eqref{LieAlg1} in the definition of Lie algebroid.
We can formalize this observation generalizing at once the definition, simply by replacing the Lie structure on $\Gamma E$ (our $\mathcal{F}$ 
in the algebraic setting) by a Loday one.
Thus, let $(\mathcal{F},[\![\;,\;]\!])$ be a left Loday algebra. Given an $X\in \mathcal{F}$, denote by 
$ad^L_X,\;ad_R^X:\mathcal{F}\rightarrow \mathcal{F}$ the endomorphisms $ad^L_X(Y)=[\![X,Y]\!],\;ad^R_X(Y)=[\![Y,X]\!]$.\\
Note that if $[\![\;,\;]\!]$ is antisymmetric, then $ad^L_X=ad^R_{-X}$
\begin{defi}\label{d1}
The pair $(\mathcal{F},[\![\;,\;]\!])$ is called a left Loday quasi-algebroid if 
$ad^L_X\in QDer_{\mathcal{R}}(\mathcal{F}),\;\forall X\in\mathcal{F}$.
\end{defi}
This amount to the condition that, given $X\in\mathcal{F},\;f\in\mathcal{A}$:
\begin{equation*}
[\![X,f.Y,]\!]-f.[\![X,Y]\!]=[ad^L_X,\mu_f](Y)=\mu_{\widehat{ad^L_X}(f)}(Y)=\widehat{ad^L_X}(f).Y\;\;\; , \forall Y\in \mathcal{F},
\end{equation*}
and motivates the following definition.
\begin{defi}\label{d2}
The mapping
$$\begin{array}{rcl}
q_{\mathcal{F}}^L:\mathcal{F}&\longrightarrow &Der_{\mathcal{R}}(\mathcal{A})\\
X&\longmapsto&q_{\mathcal{F}}^L(X):=\widehat{ad^L_X}
\end{array}$$
is called the anchor of the left Loday quasi-algebroid. If $q_{\mathcal{F}}^L$ is tensorial, it is said that 
$(\mathcal{F},[\![\;,\;]\!],q_{\mathcal{F}}^L)$ is a left Loday algebroid on $\mathcal{F}$.
\end{defi}
The condition in Definition \ref{d1} now reads
\begin{equation*}
[\![X,f.Y]\!]=f.[\![X,Y]\!]+q_{\mathcal{F}}^L(X)(f).Y,
\end{equation*}
this justifying the terminology with the ``left" prefix.
\begin{nota}
There is the corresponding notion of right Loday quasi-algebroid, when $ad^R_X\in QDer_{\mathcal{R}}(\mathcal{F})$. 
In this case, the formula reads
\begin{equation*}
[\![f.X,Y]\!]=f.[\![X,Y]\!]+q_{\mathcal{F}}^R(Y)(f).X
\end{equation*}
\end{nota}
\begin{teo}\label{t2}
Let $(\mathcal{F},[\![\;,\;]\!],q_{\mathcal{F}}^L)$ be a left Loday quasi-algebroid. Then, 
$q_{\mathcal{F}}:(\mathcal{F},[\![\;,\;]\!])\longrightarrow (Der_{\mathcal{R}}(\mathcal{A}), [\;,\;])$ 
is a morphism of left Loday $\mathcal{R}$-algebras.
\end{teo}
\begin{proof}
First, let us note that the condition of $[\![\;,\;]\!]$ being a Loday bracket on $\mathcal{F}$ means that
$$[ad_X^L,ad_Y^L]=ad_{[\![X,Y]\!]}^L\;\;\;\forall\;X,Y\in\mathcal{F}$$
To check this, let $Z\in \mathcal{F}$ and compute
\begin{eqnarray*}
[ad_X^L,ad_Y^L](Z)&=&ad_X^L(ad_Y^L(Z))-ad_Y^L(ad_X^L(Z))\\
&=&[\![X,[\![Y,Z]\!]\;]\!]-[\![Y,[\![X,Z]\!]\;]\!]\\
&=& [\![\;[\![X,Y]\!],Z]\!]+[\![Y,[\![X,Z]\!]\;]\!]-[\![Y,[\![X,Z]\!]\;]\!]\\
&=& [\![\;[\![X,Y]\!],Z]\!]\\
&=& ad_{[\![X,Y]\!]}^L(Z).
\end{eqnarray*}
As this is valid for all $Z\in\mathcal{F}$, we get the stated equivalence.\\
Now, Corollary \ref{c1} says that $\;\forall\; X,Y\in \mathcal{F}$:
$$[q_{\mathcal{F}}^L(X),q_{\mathcal{F}}^L(Y)]=[\widehat{ad _X^L},\widehat{ad _Y^L}]=
\widehat{[ad_X^L,ad_Y^L]}=\widehat{ad_{[\![X,Y]\!]}^L}=q_{\mathcal{F}}^L([\![X,Y]\!]).$$
\end{proof}
\begin{nota}
This result partly answers a question raised in Remark 3.3 (1) of \cite{Stienon}.
\end{nota}
The definitions just given can be particularized to the case of Lie algebras (i.e, $[\![\;,\;]\!]$ antisymmetric).
\begin{defi}\label{d3}
Let $(\mathcal{F},[\![\;,\;]\!])$ be a Lie algebra. If $ad_X^L\in QDer_{\mathcal{R}}(\mathcal{F}), \;\forall X\in\mathcal{F}$,
 we say that $(\mathcal{F},[\![\;,\;]\!],q_{\mathcal{F}})$, is a Lie quasi-algebroid, where
$$\begin{array}{rcl}
q_{\mathcal{F}}:F&\longrightarrow &Der_{\mathcal{R}}(\mathcal{A})\\
X&\longmapsto&q_{\mathcal{F}}(X):=\widehat{ad^L_X}
\end{array}$$
is the anchor map. If $q_{\mathcal{F}}$ is tensorial ($\mathcal{A}$-linear), then we say that
$(\mathcal{F},[\![\;,\;]\!],q_{\mathcal{F}})$ is a Lie algebroid.
\end{defi}
\begin{nota}
Note that in this case the distinction between the left and right cases is irrelevant: each left Lie quasi-algebroid with anchor 
$q_{\mathcal{F}}$ is also a right Lie quasi-algebroid with anchor $-q_{\mathcal{F}}$.
\end{nota}
How different are left (and right) Loday quasi-algebroids, Lie quasi-algebroids and Loday algebroids? In some cases, there is no such 
distinction: If we take $\mathcal{R}=\mathbb{R},\;\mathcal{A}=\mathcal{C}^{\infty}(M),\;\mathcal{F}=\Gamma E$, with $\pi:E\rightarrow M$ a 
vector bundle over a manifold $M$, Grabowski calls a $QD-Loday$ (resp. Lie) algebroid a left Loday 
(resp. Lie) quasi-algebroid (that is, 
$ad_X^L\in QDer_{\mathcal{R}}(\mathcal{F}),\;\forall X\in\mathcal{F}$) such that 
$ad_X^R\in QDer_{\mathcal{R}}(\mathcal{F}),\;\forall X\in\mathcal{F}$; then, he proves:
\begin{teo}\label{teograb1}
Every QD-Loday algebroid (resp. QD-Lie) with $rank\geq 1$, is a Loday algebroid (resp. Lie).
\end{teo}
\begin{teo}\label{teograb2}
Every QD-Loday algebroid of rank 1, is a QD-Lie algebroid.
\end{teo}

\section{Generation of Loday algebroids}
As we have seen in the previous section, in order to get genuine examples of Loday quasi-algebroids, 
we must avoid that the two conditions $ad^L_X\in QDer_{\mathcal{R}}(\mathcal{F})$ and 
$ad^R_X\in QDer_{\mathcal{R}}(\mathcal{F})$ be satisfied \emph{simultaneously}. To get examples of this
situation, it is useful to know how to generate Loday brackets from operators with certain features.
First of all, note that given a left Loday bracket
$[\![\;,\;]\!]:\mathcal{F}\times\mathcal{F}\rightarrow\mathcal{F}$, if we define
\begin{eqnarray*}
[\![\;,\;]\!]':\mathcal{F}\times\mathcal{F}&\longrightarrow&\mathcal{F}\\
(X,Y)&\longmapsto& [\![X,Y]\!]':=[\![Y,X]\!],
\end{eqnarray*}
then we have that $[\![\;,\;]\!]'$ is $\mathcal{R}-$bilinear and for all $X,Y,Z\in\mathcal{F}$:
\begin{eqnarray*}
[\![X,[\![Y,Z]\!]'\;]\!]'+[\![\;[\![X,Z]\!]',Y]\!]'&=& [\![X,[\![Z,Y]\!]\;]\!]'+[\![\;[\![Z,X]\!],Y]\!]' \\
&=& [\![\;[\![Z,Y]\!],X]\!]+[\![Y,[\![Z,X]\!]\;]\!]\\
&=&[\![Z,[\![Y,X]\!]\;]\!]\\
&=& [\![\;[\![X,Y]\!]',Z]\!]',
\end{eqnarray*}
thus, $[\![\;,\;]\!]'$ is a right Loday bracket.\\
Analogously, given a right Loday bracket we can define a left Loday one, obtaining a correspondence 
between left and right Loday algebras.

\begin{pro}If $(\mathcal{A},\cdot)$ is an associative $\mathcal{R}-$algebra\footnote{That is, $\mathcal{A}$ is an $\mathcal{R}-$module
endowed with an associative product $\cdot :\mathcal{A}\times\mathcal{A}\rightarrow\mathcal{A}$} and, moreover,
is endowed with an $\mathcal{R}-$linear mapping
 $D:\mathcal{A}\rightarrow\mathcal{A}$ verifying
$$D(a\cdot (D(b)))=D(a)\cdot D(b)=D((D(a))\cdot b),\;\;\;\forall a,b\in\mathcal{A},$$
then we can define:
\begin{eqnarray*}
[\;,\;]:\mathcal{A}\times\mathcal{A}&\longrightarrow&\mathcal{A}\\
(a,b)&\longmapsto&[a,b]:=D(a)\cdot b-b\cdot D(a),
\end{eqnarray*}
which satisfies the properties of $\mathcal{R}-$bilinearity and the left Leibniz rule (so, it is a left Loday algebra).
\end{pro}
\begin{proof}
Let us check first the $\mathcal{R}-$bilinearity
\begin{eqnarray*}
[\alpha a+\beta b, c]&=&D(\alpha a+\beta b)\cdot c-c\cdot D(\alpha a+\beta b)\\
&=& \alpha D(a)\cdot c+\beta D(b)\cdot c-\alpha c\cdot D(a)-\beta c\cdot D(b)\\
&=& \alpha (D(a)\cdot c-c\cdot D(a))+\beta (D(b)\cdot c-c\cdot D(b))\\
&=& \alpha [a,c]+\beta [b,c]
\end{eqnarray*}
\begin{eqnarray*}
[a,\beta b+\gamma c]&=&D(a)\cdot (\beta b+\gamma c)-(\beta b+\gamma c)\cdot D(a)\\
&=& \beta D(a)\cdot b+\gamma D(a)\cdot c-\beta b\cdot D(a)-\gamma c\cdot D(a)\\
&=& \beta (D(a)\cdot b-b\cdot D(a))+\gamma (D(a)\cdot c-c\cdot D(a))\\
&=& \beta [a,b]+\gamma [a,c]
\end{eqnarray*}
For the left Leibniz rule, we have:
\begin{eqnarray*}
&&[a,[b,c]\;]-[\;[a,b],c]-[b,[a,c]\;]\\
&=&D(a)\cdot [b,c]-[b,c]\cdot D(a)\\
&&-D([a,b])\cdot c+c\cdot D[a,b]\\
&&-D(b)\cdot [a,c]+[a,c]\cdot D(b)\\
&=& D(a)\cdot (D(b)\cdot c-c\cdot D(b))-(D(b)\cdot c-c\cdot D(b))\cdot D(a)\\
&&-D(D(a)\cdot b-b\cdot D(a))\cdot c+c\cdot D(D(a)\cdot b-b\cdot D(a))\\
&&-D(b)\cdot (D(a)\cdot c-c\cdot D(a))+(D(a)\cdot c-c\cdot D(a))\cdot D(b)\\
&=& D(a)\cdot D(b)\cdot c-D(a)\cdot c\cdot D(b)-D(b)\cdot c\cdot D(a)\\
&& +c\cdot D(b)\cdot D(a)-D(a)\cdot D(b)\cdot c+D(b)\cdot D(a)\cdot c\\
&& +c\cdot D(a)\cdot D(b)-c\cdot D(b)\cdot D(a)-D(b)\cdot D(a)\cdot c\\
&&+D(b)\cdot c\cdot D(a)+ D(a)\cdot c\cdot D(b)-c\cdot D(a)\cdot D(b)=0
\end{eqnarray*}
\end{proof}

\begin{ejem}Some examples of such mappings $D:\mathcal{A}\rightarrow\mathcal{A}$ are:
\begin{enumerate}[(a)]
\item The identity $D=Id$. In this particular case we obtain a Lie algebra.
\item\label{diff} A zero-square derivation $D$. Indeed, if this is the case,
$$D(a\cdot D(b))=D(a)\cdot D(b)+a\cdot D^2(b)= D(a)\cdot D(b)\;\;\;\forall a,b\in\mathcal{A}.$$
\item A projector $D$, that is, $D$ is an algebra morphism and $D^2=D$. Then:
$$D(a\cdot D(b))=D(a)\cdot D^2(b)=D(a) \cdot D(b)\;\;\;\forall a,b\in\mathcal{A}.\\$$
\end{enumerate}
\end{ejem}
Now, we can give a simple example of a left Loday quasi-algebroid which does not admit a right Loday quasi-algebroid structure.
\begin{ejem}\label{e1}
Consider $\mathcal{F}=\Omega(\mathbb{R}^6)$, which is an $\mathbb{R}$-algebra with the exterior product $\wedge$ and, moreover, a $\mathcal{C}^{\infty}(\mathbb{R}^6)$-module (i.e $\mathcal{R}=\mathbb{R},\,\mathcal{A}=\mathcal{C}^{\infty}(M)$). Define:
$$[\![\alpha,\beta]\!]=d(\alpha)\wedge \beta-\beta\wedge d(\alpha)=(1-(-1)^{|\beta|\!|\alpha+1|})d(\alpha)\wedge\beta.$$
It is immediate that $(\Omega(\mathbb{R}^6),[\![\;,\;]\!])$ is a left Loday algebra, as $d$ is a square-zero operator (see \eqref{diff} above). Now, we have, for any $\alpha,\,\beta\in\Omega(\mathbb{R}^6),\;f\in\mathcal{C}^{\infty}(\mathbb{R}^6)$:
\begin{eqnarray*}
[\![\alpha, f.\beta]\!] & = & (1-(-1)^{|\beta|\!|\alpha+1|})d(\alpha)\wedge f.\beta\\
&=& f.(1-(-1)^{|\beta|\!|\alpha+1|})d(\alpha)\wedge\beta \\
&=& f.[\![\alpha, \beta]\!]
\end{eqnarray*}
Thus, on the other hand, if $f=x_6$, $\alpha=dx_1\wedge dx_2\wedge dx_3\wedge dx_4$, $\beta=dx_5$:
\begin{eqnarray*}
[\![f.\alpha,\beta]\!]&=&(1-(-1)^{|1|\!|5|})d(x_6.dx_1\wedge dx_2\wedge dx_3\wedge dx_4)\wedge dx_5\\
&=& -2 dx_1\wedge dx_2\wedge dx_3\wedge dx_4\wedge dx_5\wedge dx_6
\end{eqnarray*}
but
$$
f.[\![\alpha,\beta]\!]= (1-(-1)^{|1|\!|5|})x_6.d(dx_1\wedge dx_2\wedge dx_3\wedge dx_4)\wedge dx_5 =0
$$
and there is no $q^L_{\Omega(\mathbb{R}^6)}:\Omega(\mathbb{R}^6)\longrightarrow \mathfrak{X}(\mathbb{R}^6)$ such that
$$q^L_{\Omega(\mathbb{R}^6)}(\alpha)(f).dx_5=-2dx_1\wedge dx_2\wedge dx_3\wedge dx_4\wedge dx_5\wedge dx_6.$$
Thus, $(\Omega(\mathbb{R}^6),[\![\;,\;]\!])$ has a left Loday quasi-algebroid structure, with anchor $q^L_{\Omega(\mathbb{R}^6)}\equiv 0$, but it does not admit a right Loday quasi-algebroid structure. Note that $q^L_{\Omega(\mathbb{R}^6)}$ is (trivially) tensorial.
\end{ejem}

The following, less trivial, example was suggested to us by Y. Sheng. It shows that the kind of structures we are considering can appear in the more general context of
higher order Courant algebroids (although here we just take $n=1$ for simplicity) through the associated Dorfman bracket, see \cite{Sheng1}.
\begin{ejem}
Let $M$ be a differential manifold. Consider the vector bundle $TM\oplus T^*M$ whose sections are endowed with the Dorfman bracket:
$$\begin{array}{cccc}
[\![\;,\;]\!]:  & TM\oplus T^*M \times TM\oplus T^*M  & \longrightarrow & TM\oplus T^*M\\
   & (X+\alpha,Y+\beta)&\longmapsto & [X,Y]+\mathcal{L}_X\beta-\iota_Yd\alpha\\
\end{array}$$
Then we have a left Loday algebra, as $[\![\;,\;]\!]$ is clearly $\mathbb{R}-$bilinear and
\begin{equation*}\label{ll}
\begin{split}
& [\![ [\![X+\alpha, Y+\beta ]\!],Z+\gamma ]\!]+[\![ Y+\beta, [\![ X+\alpha, Z+\gamma ]\!] ]\!]\\
= & [[X,Y],Z]+\mathcal{L}_{[X,Y]}\gamma -\iota _Zd(\mathcal{L}_X\beta -\iota _Yd\alpha )+[Y,[X,Z]]+\mathcal{L}_Y(\mathcal{L}_X\gamma -\iota _Zd\alpha )-\iota _{[X,Z]}d\beta \\
= & [X,[Y,Z]]+\mathcal{L}_{[X,Y]}\gamma+\mathcal{L}_Y(\mathcal{L}_X\gamma) -\iota _Zd(\mathcal{L}_X\beta)-\iota _{[X,Z]}d\beta+\iota _Zd(\iota _Yd\alpha )-\mathcal{L}_Y(\iota _Zd\alpha )\\
= & [X,[Y,Z]]+\mathcal{L}_X(\mathcal{L}_Y\gamma)-\mathcal{L}_{X}\iota _Zd\beta-\iota _{[Y,Z]}d\alpha \\
= & [X,[Y,Z]]+\mathcal{L}_X((\mathcal{L}_Y\gamma)-\iota _Zd\beta)-\iota _{[Y,Z]}d\alpha \\
= & [\![ X+\alpha,[\![ Y+\beta,Z+\gamma ]\!] ]\!],
\end{split}
\end{equation*}
for all $X+\alpha,\; Y+\beta,\;Z+\gamma\in TM\oplus T^*M$. Moreover, $ad_{X+\alpha}^L\in QDer_{\mathbb{R}}(TM\oplus T^*M)$: Let $f\in \mathcal{C}^{\infty}(M)$ and $Y+\beta\in TM\oplus T^*M$. Then,
\begin{equation*}\label{ll}
\begin{split}
[ad_{X+\alpha}^L , \mu _f](Y+\beta)&=  [\![ X+\alpha,f.(Y+\beta)]\!]-f.[\![ X+\alpha,Y+\beta ]\!]\\
&=[X, f.Y]+ \mathcal{L}_X(f\alpha)-\iota_{f.Y}d\alpha-f[X,Y]-f\mathcal{L}_X\beta +f \iota _yd\alpha\\
&=X(f).Y + \mathcal{L}_X(f\beta)-f\mathcal{L}_X\beta + f\iota _Yd\alpha -\iota _{fY}d\alpha\\
&=X(f).Y+X(f).\beta\\
&=X(f).(Y+\beta)\\
&=\mu_{X(f)}(Y+\beta)\\
&=\widehat{ad_{X+\alpha}^L}(f).(Y+\beta),
\end{split}
\end{equation*}
so $(TM\oplus T^*M, [\![\;,\;]\!])$ is a left Loday quasi-algebroid, with anchor map the projection onto the first factor:
$$ q^L_{TM\oplus T^*M}(X+\alpha)=\widehat{ad_{X+\alpha}^L}=X. $$
Note that in this case the anchor is tensorial: If $f,g\in\mathcal{C}^{\infty}(M)$ and $X+\alpha\in TM\oplus T^*M$, then
\begin{equation*}
[q^L_{TM\oplus T^*M},\mu_f](X+\alpha)(g)=\widehat{ad_{f.(X+\alpha)}^L}(g)-f.\widehat{ad_{X+\alpha}^L}(g)=0,
\end{equation*} 
so $(TM\oplus T^*M,[\![\;,\;]\!])$ is indeed a left Loday algebroid. However, for $ad_{X+\alpha}^R$ we find:
\begin{equation*}\label{ll}
\begin{split}
[ad_{X+\alpha}^R , \mu _f](Y+\beta)&=  B(f(Y+\beta),X+\alpha)-f.B(Y+\beta,X+\alpha)\\
&=[f.Y,X]+\mathcal{L}_{f.Y}\alpha -\iota _Xd(f.\beta)-f.[X,Y]-f\mathcal{L}_Y\alpha +f\iota _Xd\beta\\
&=-X(f).Y+\mathcal{L}_{f.Y}\alpha -\iota _Xd(f.\beta)-f\mathcal{L}_Y\alpha +f\iota _Xd\beta
\end{split}
\end{equation*}
and the term $\mathcal{L}_{f.Y}\alpha -\iota _Xd(f.\beta)-f\mathcal{L}_Y\alpha +f\iota _Xd\beta$ clearly spoils the
possibility that $ad_{X+\alpha}^R$ be a quasi-derivation.
\end{ejem}

\section{Left omni-Loday algebroids and omni-Lie algebroids}
Having established  the non-triviality of left Loday quasi-algebroids, we now turn to the question of whether an analogue of Weinstein's omni-Lie algebra exists for these structures.
As before, let $\mathcal{A}$  be an associative algebra, commutative and with unit element $1_{\mathcal{A}}$ over a ring $\mathcal{R}$, commutative and with unit element $1_{\mathcal{R}}$. Also, let $\mathcal{F}$ a faithful $\mathcal{A}$-module.
\begin{defi}\label{d4}
Consider the product space $\mathfrak{gl}(\mathcal{F})\times\mathcal{F}$ and define the bracket
\begin{eqnarray*}
\{\;,\;\}:(\mathfrak{gl}(\mathcal{F})\times\mathcal{F})\times (\mathfrak{gl}(\mathcal{F})\times\mathcal{F})&\longrightarrow&\mathfrak{gl}(\mathcal{F})\times\mathcal{F}\\
((\Phi ,X),(\Psi ,Y))&\longmapsto&\{(\Phi ,X),(\Psi ,Y)\}:=([\Phi ,\Psi ], \Phi Y)
\end{eqnarray*}
where $[\;,\;]$ is he commutator of endomorphisms.
\end{defi}
\begin{nota}
It is straightforward to check that $\{\;,\;\}$ is $\mathcal{R}$-bilinear. However, it does not satisfy Jacobi's identity (here
$\circlearrowleft$ denotes ciclic sum), as we have:
$$\circlearrowleft\{(\Phi ,X),\{(\Psi ,Y),(\Upsilon ,Z)\}\}=(0, \Phi (\Psi Z)+\Upsilon (\Phi Y)+\Psi (\Upsilon X)).$$
As stated in the introduction, the bracket $\{\;,\;\}$ satisfies instead the left Leibniz identity:
$$
\{(\Phi ,X),\{(\Psi ,Y),(\Upsilon ,Z)\}\}=\{\{(\Phi ,X),(\Psi ,Y)\},(\Upsilon ,Z)\}+\{(\Psi ,Y),\{(\Phi ,X),(\Upsilon ,Z)\}\}
$$
\end{nota}
Now, let $B:\mathcal{F}\times\mathcal{F}\longrightarrow\mathcal{F}$ be an $\mathcal{R}$-bilinear form. Define the ``graph" of B as
\begin{equation*}
\mathcal{F}_B:=\{(ad^L_X,X):X\in\mathcal{F}\}\subseteq \mathfrak{gl}(\mathcal{F})\times\mathcal{F},
\end{equation*}
where
\begin{eqnarray*}
ad_X^L:\mathcal{F}&\longrightarrow&\mathcal{F}\\
Y&\longmapsto&B(X,Y)
\end{eqnarray*}
\begin{pro}\label{p1}
The graph $\mathcal{F}_B$ is closed under $\{\;,\;\}$ if and only if $(\mathcal{F},B)$ is a left Loday algebra. Moreover, if $B$ 
is antisymmetric, $\mathcal{F}_B$ is closed if and only if $(\mathcal{F},B)$ is a Lie algebra.
\end{pro}
\begin{proof}
$\mathcal{F}_B$ is closed with respect to $\{\;,\;\}$ if, and only if, for all $X,Y\in \mathcal{F}$ we have:
$$\{(ad_X^L,X)\;,\;(ad_Y^L,Y)\}=([ad_X^L,ad_Y^L]\;,\;ad_X^L(Y))=(ad_{B(X,Y)}^L,B(X,Y)),$$
that is, for all $Z\in \mathcal{F}$:
$$[ad_X^L,ad_Y^L](Z)=ad_{B(X,Y)}^L(Z),$$
or,
$$B(X,B(Y,Z))-B(Y,B(X,Z))=B(B(X,Y),Z)\;\;\;\;\;\forall \;X,Y,Z\in\mathcal{F},$$
or equivalently,
$$B(B(X,Y),Z)+B(Y,B(X,Z))=B(X,B(Y,Z))\;\;\;\;\;\forall \;X,Y,Z\in\mathcal{F},$$
which is the left Leibniz identity, i.e, $(\mathcal{F},B)$ is a left Loday algebra.
\end{proof}
For left Loday quasi-algebroids, we have the following.
\begin{teo}\label{t3}
Let $B:\mathcal{F}\times\mathcal{F}\longrightarrow\mathcal{F}$ be an $\mathcal{R}$-bilinear form and $\rho:\mathcal{F}\longrightarrow Der_{\mathcal{R}}(\mathcal{A})$ a morphism of $\mathcal{R}$-modules. Then, $(\mathcal{F},B)$ is a left Loday quasi-algebroid with anchor map $\rho$, if and only if $\mathcal{F}_B$ is closed with respect to $\{\;,\;\}$ and $B$ is such that
\begin{equation*}
B(X,f.Z)=f.B(X,Z)+\rho(X)(f).Z,
\end{equation*}
(that is, $ad^L_X\in QDer_{\mathcal{R}}(\mathcal{F})$).
\end{teo}
\begin{proof}
If $(\mathcal{F},\;B)$ is a left Loday quasi-algebroid, it is also a left Loday algebra and then,
by Proposition \ref{p1}, $\mathcal{F}_B$ is closed under $\{\;,\;\}$. 
On the other hand, the condition of being quasi-algebroid implies that for all $X,Z\in \mathcal{F}$ and
for all $f\in \mathcal{A}$, we have
\begin{equation*}
[ad_X^L,\mu_f](Z)=\rho (X)(f).Z,
\end{equation*}
that is,
\begin{equation*}
B(X,f.Z)=f.B(X,Z)+\rho (X)(f).Z.
\end{equation*}
For the second implication, consider $\mathcal{F}_B$ closed with respect to $\{\;,\;\}$, so $(\mathcal{F},B)$ is 
a left Loday algebra (see Proposition \ref{p1}).
Moreover, for all $X,Z\in \mathcal{F}$ and for all $f\in \mathcal{A}$:
\begin{eqnarray*}
[ad_X^L,\mu_{f}](Z)&=&B(X,f.Z)-f.B(X,Z)\\
&=&f.B(X,Z)+\rho(X)(f).Z-f.B(X,Z)\\
&=&\rho(X)(f).Z\\
&=&\mu_{\rho(X)(f)}(Z),
\end{eqnarray*}
so $$[ad_X^L,\mu_f]=\mu_{\rho(X)(f)}$$
that is, $ad_X^L$ is a quasi-derivation for all $X\in \mathcal{F}$.
Thus, $(\mathcal{F},\;B)$ is a left quasi-algebroid with anchor map $\rho$.
\end{proof}
Let us try to get rid of the ``quasi'' prefix.
\begin{teo}\label{t4}
Let $B:\mathcal{F}\times\mathcal{F}\longrightarrow\mathcal{F}$ be an $\mathcal{R}$-bilinear form and $\rho:\mathcal{F}\longrightarrow Der_{\mathcal{R}}(\mathcal{A})$ a morphism of $\mathcal{R}$-modules. Suposse that $\mathcal{F}_B$ is closed with respect to $B$ and that $B$ is such that
\begin{enumerate}[(a)]
\item\label{conda} $B(X,f.Z)=f.B(X,Z)+\rho(X)(f).Z$
\item\label{condb} $B(f.X,Z)=f.B(X,Z)-\rho(Z)(f).X$
\end{enumerate}
(that is, $ad^L_X\in QDer_{\mathcal{R}}(\mathcal{F})$). Then $(\mathcal{F},B,\rho)$ is a left Loday algebroid.
\end{teo}
\begin{proof}
We only need to prove that $\rho$ is tensorial or, equivalenty, that for all $X,Z\in \mathcal{F}$ and for all $g,f\in \mathcal{A}$
the following holds:
\begin{equation*}
B(f.X,g.Z)-g.B(f.X,Z)=f.B(X,g.Z)-fg.B(X,Z).
\end{equation*}
But the hypothesis in the statement guarantee that:
\begin{eqnarray*}
B(f.X,g.Z)-g.B(f.X,Z)&=&f.B(X,g.Z)-g\rho(Z)(f).X-g.B(f.X,Z)\\
&=& fg.B(X,Z)+f\rho(X)(g).Z-g\rho(Z)(f).X\\
&&-fg.B(X,Z)+g\rho(Z)(f).X\\
&=& f\rho(X)(g).Z
\end{eqnarray*}
and
\begin{eqnarray*}
f.B(X,g.Z)-fg.B(X,Z)&=& fg.B(X,Z)+f\rho(X)(g).Z-fg.B(X,Z)\\
&=& f\rho(X)(g).Z.
\end{eqnarray*}
Thus, $(\mathcal{F},\;B,\;\rho)$ is a left Loday algebroid.
\end{proof}
\begin{nota}
However, we can not say anything about the converse, as the Example \ref{e1}
shows (there, we have a left Loday algebroid and the first condition \eqref{conda} above is trivially satisfied while \eqref{condb} is not).
\end{nota}
We can avoid the ``quasi'' prefix if we add the condition of antisymmetry to $B$, thus entering into the realm of Lie structures.
\begin{teo}\label{t5}
Let $B:\mathcal{F}\times\mathcal{F}\longrightarrow\mathcal{F}$ be an $\mathcal{R}$-bilinear form, and $\rho:\mathcal{F}\longrightarrow Der_{\mathcal{R}}(\mathcal{A})$ a morphism of $\mathcal{R}$-modules. Then, $(\mathcal{F},B,\rho)$ is a Lie algebroid if and only if $\mathcal{F}_B$ is closed with respect to $\{\;,\;\}$, $B$ is antisymmetric and, for all $X,Z\in\mathcal{F}$ and $f\in\mathcal{A}$, the following holds:
\begin{equation*}
B(X,f.Z)=f.B(X,Z)+\rho(X)(f).Z
\end{equation*}
(that is, $ad^L_X\in QDer_{\mathcal{R}}(\mathcal{A})$).
\end{teo}
\begin{proof}
If $(\mathcal{F},\;B,\;\rho)$ is a Lie algebroid, $(\mathcal{F},\;B)$ is a Lie algebra, that is, 
$(\mathcal{F},\;B)$ is a left (and right) Loday algebra and $B$ is antisymmetric, so Proposition \ref{p1}
tells us that $\mathcal{F}_B$ is closed with respect to $\{\;,\;\}$. Now, let $X,Z\in \mathcal{F},\;f\in \mathcal{A}$;
then we have
 \begin{equation*}
 [ad_X^L,\mu_f](Z)=\mu_{\rho(X)(f)}(Z),
 \end{equation*}
 which is the same as
 \begin{equation*}
B(X,f.Z)=f.B(X,Z)+\rho(X)(f).Z.
 \end{equation*}
 If now is $\mathcal{F}_B$ closed with respect to $\{\;,\;\}$, Proposition \ref{p1} again tells us that
 $(\mathcal{F},\;B)$ is a left Loday algebra, but as $B$ is also antisymmetric, 
 $(\mathcal{F},\;B)$ is a Lie algebra.\\
 On the other hand, the hypothesis of Theorem \ref{t3} are satisfied, so we know that $ad_X$ is 
 a quasi-derivation for all $X\in \mathcal{F}$ and $(\mathcal{F},\;B)$ is a Lie quasi-algebroid
 with anchor map $\rho$.\\
 To finish, let us check (see Theorem \ref{t4}) that for all $X,Y,Z\in \mathcal{F},\;\forall\;f\in \mathcal{A}$ the
 following holds:
 $$B(f.X,Z)=f.B(X,Z)-\rho(Z)(f).X.$$
 But, by the antisymmetry of $B$:
\begin{eqnarray*}
 B(f.X,Z)&=& -B(Z,f.X)\\
 &=&-f.B(Z,X)-\rho(Z)(f).X\\
 &=& f.B(X,Z)-\rho(Z)(f).X.
 \end{eqnarray*}
 So $(\mathcal{F},\;B,\;\rho)$ is a Lie algebroid.
\end{proof}
The preceding results motivate the following definition.
\begin{defi}
Let $\mathcal{A}$ be an associative, commutative algebra with unit element $1_{\mathcal{A}}$ over a commutative ring with 
unit element $1_{\mathcal{R}}$. Let $\mathcal{F}$ be an $\mathcal{A}$-module. We call $(\mathfrak{gl}(\mathcal{F})\times \mathcal{F},\{\;,\;\})$ the left omni-Loday algebroid determined by
$\mathcal{F}$.
\end{defi}
\begin{nota}
Note that if $(\mathcal{F},B,\rho)$ is a left Loday algebroid then, in particular, is a left Loday quasi-algebroid and thus 
$\mathcal{F}_B\subset \mathfrak{gl}(\mathcal{F})\times\mathcal{F}$ is closed with respect to $\{\;,\;\}$, by Theorem \ref{t3}: every left Loday algebroid can be seen as a closed subespace of left omni-Loday algebroid.
\end{nota}
\begin{nota}
In the case of Lie algebroids, we have the same situation as in the preceding remark: given an $\mathcal{R}$-bilinear $\mathcal{F}$-valued form $B:\mathcal{F}\times\mathcal{F}\longrightarrow\mathcal{F}$ such that it is antisymmetric and satisfies $ad^L_X\in QDer_{\mathcal{R}}(\mathcal{F})$, by Theorem \ref{t5}
there is a correspondence between Lie algebroids $(\mathcal{F},B,\rho)$ and closed subspaces $\mathcal{F}_B$, but this time given by an ``if and only if" statement. Thus, we could call $(\mathfrak{gl}(\mathcal{F})\times\mathcal{F},\{\;,\;\})$ an omni-Lie algebroid as well.
\end{nota}
It is worth noting that a different definition for omni-Lie algebroids (based on the notion of Courant structures on the direct sum of
the gauge Lie algebroid and the bundle of jets of a vector bundle $E$ over a manifold $M$), has been presented very recently in \cite{Chen}. It would be interesting to know if this definition is equivalent to ours.

\label{lastpage}
\end{document}